\documentclass[11pt]{article}
\usepackage[T1]{fontenc}
\usepackage{latexsym,amssymb,amsmath,amsfonts,amsthm}
\usepackage{graphicx}
\usepackage{subfigure}
\usepackage{placeins}
\usepackage{color}
\usepackage{indentfirst}
\numberwithin{equation}{section}
\newcommand{\R}{{\mathbb R}}

\newcommand{\C}{{\mathbb C}}
\newcommand{\s}{{\mathbb S}}

\newcommand{\LL}{{\mathcal L}}

\newcommand{\be}{\begin{eqnarray}}
\newcommand{\ben}{\begin{eqnarray*}}
\newcommand{\en}{\end{eqnarray}}
\newcommand{\enn}{\end{eqnarray*}}

\newcommand{\curl}{{\rm curl\,}}

\newcommand{\grad}{{\rm grad\,}}
\newcommand{\divv}{{\rm div\,}}

\newcommand{\Om}{\Omega}

\newtheorem{theorem}{Theorem}[section]
\newtheorem{lemma}[theorem]{Lemma}
\newtheorem{corollary}[theorem]{Corollary}
\newtheorem{definition}[theorem]{Definition}
\newtheorem{remark}[theorem]{Remark}
\newtheorem{proposition}[theorem]{Proposition}

\definecolor{rot}{rgb}{0.000,0.000,0.000}
\definecolor{blau}{rgb}{0,0,1}
\newcommand{\rot}{\textcolor[rgb]{0,0,0}  }

\begin{document}
\renewcommand{\theequation}{\arabic{section}.\arabic{equation}}
\begin{titlepage}

\title{\bf {\color{rot}{Uniqueness and factorization method for inverse elastic scattering  with a single incoming wave}}}


\author{
Johannes Elschner, \thanks{Weierstrass Institute (WIAS), Mohrenstr. 39, 10117 Berlin, Germany ({\tt elschner@wias-berlin.de}).
}\qquad
 Guanghui Hu\thanks{Beijing Computational Science Research Center (CSRC), 100193 Beijing, China ({\tt
hu@csrc.ac.cn}).} \\ [4mm]
}

\date{}
\end{titlepage}
\maketitle
\vspace{.2in}
\begin{abstract} The first part of this paper is concerned with the uniqueness to inverse time-harmonic elastic scattering from bounded rigid obstacles
in two dimensions. It is proved that a connected polygonal obstacle can be uniquely identified by the far-field pattern
corresponding to a single incident plane wave. Our approach is based on a new reflection principle for the first boundary value problem of the Navier equation.
In the second part, we propose a revisited factorization method to recover a rigid elastic body with a single far-field pattern.

\vspace{.2in} {\bf Keywords}: Uniqueness, inverse elastic scattering, rigid polygonal obstacle, single plane wave, reflection principle.  \end{abstract}

\section{Introduction and main results}
Let $D\subset \R^2$ be a bounded elastic body such that its exterior $D^c:=\R^2\backslash\overline{D}$ is connected, and let $D^c$ be occupied by a homogeneous and isotropic elastic medium. Suppose that
a time-harmonic elastic plane wave of the form
\be\label{planewave}
u^{in}(x)=c_p\, d\,e^{ik_px\cdot d}+c_s\, d^\perp e^{ik_sx\cdot d}
\en
is incident on the scatterer $D$.
Here, $d=(\cos\theta, \sin\theta)^T \, , \;\theta\in[0, 2\pi)$, is the incident direction, $d^\perp:=(-\cos\theta,\sin\theta)^T$ is orthogonal to $d$,  $\omega>0$ is the frequency and $k_p:=\omega/\sqrt{\lambda+2\mu}$, $k_s:=\omega/\sqrt{\mu}$ are the compressional and shear wave numbers, respectively, and {\color{rot}{$c_p, c_s\in \C$ satisfy $|c_p|+|c_s|\neq 0$.}}
Note that for simplicity the density of the background medium has been normalized to be one and the Lam\'e constants $\lambda$ and $\mu$ satisfy $\mu>0$ and $\lambda+2\mu>0$ in two dimensions.
The propagation of time-harmonic elastic waves in $D^c$ is governed by the Navier equation (or system)
\be\label{Na}
\LL_\omega u:=\mu \Delta+(\lambda+\mu)\nabla (\nabla\cdot u)+\omega^2u=0\quad\mbox{in}\quad D^c,\quad {\color{rot} {u=(u_1, u_2)^T,}}
\en
where $u=u^{in}+u^{sc}$ denotes the total displacement field.
By Hodge decomposition, any solution $u$ to (\ref{Na}) can be decomposed into the form
\be\label{decomposition}
u=u_p+u_s,\quad u_p:=(-1/k_p^2)\, \grad\divv u,\quad u_s:=(1/k_s^2)\,\curl\overrightarrow{\curl} u,
\en
where $u_p$ and $u_s$ {\color{rot}{are called}} compressional and shear waves, respectively. Note that in (\ref{decomposition}) the two $\curl$ operators are defined as
\be\label{curl}\overrightarrow{\curl} u:=\partial_2 u_1-\partial_1 u_2,\qquad \curl f=(-\partial_2 f, \partial_1 f)^T.
\en
Moreover, $u_\alpha\, (\alpha=p, s$) satisfies the vector Helmholtz equation $(\Delta+k_\alpha^2) u_\alpha=0$ and
$\overrightarrow{\curl} u_p=\divv u_s=0$ in $D^c$.
 Obviously, the scattered field $u^{sc}$ also satisfies the Navier equation (\ref{Na}) in $D^c$. In this paper we require $u^{sc}$ to fulfill the Kupradze radiation condition defined as follows.
\begin{definition}\label{Def}
The scattered field $u^{sc}$ to (\ref{Na}) is called a Kupradze radiating solution if its compressional and shear parts $u^{sc}_\alpha$ ($\alpha=p,s$) satisfy the Sommerfeld radiation condition for the vector Helmholtz equation, i.e.,
\ben
\partial_r u^{sc}_\alpha-ik_\alpha u^{sc}_\alpha=o(r^{-1/2})\quad \mbox{as}\quad r=|x|\rightarrow\infty,
\enn
uniformly in all directions $\hat{x}=x/|x|$  {\color{rot}{on the unit circle  $\mathbb{S}:=\{x\in\R^2: |x|=1\}$.}}
\end{definition}
It is well known that the forward scattering problem admits a unique solution $u\in H^1_{loc}(D^c)$. To prove existence of solutions we refer to \cite[Chapter 7.3]{Kupradze} for the standard integral equation method applied to rigid scatterers with $C^2$-smooth boundaries and to a recent paper \cite{BHSY} using the variational approach for treating Lipschitz boundaries. This paper is concerned with the inverse scattering problem of recovering $\partial D$ from the information of the scattered field of a single incoming plane wave.
To state the inverse problem, we need to define the far-field pattern of the scattered field.
It is well known that the compressional and shear parts $u^{sc}_\alpha$ ($\alpha=p,s$) of a radiating solution $u^{sc}$
to the Navier equation have an asymptotic behavior of the form \cite{HH, Kupradze, AK}
\ben
u^{sc}_p(x)=\frac{e^{ik_pr}}{ \sqrt{r}}\left\{
u^{\infty}_p(\hat{x})\,\hat{x}+O(\frac{1}{r})\right\},\\
u^{sc}_s(x)=\frac{e^{ik_sr}}{ \sqrt{r}}\left\{
u^{\infty}_s(\hat{x})\,\hat{x}^\perp+O(\frac{1}{r})\right\}
\enn
as $r=|x|\rightarrow\infty$, where $u^\infty_p$ and $u^\infty_s$ are both scalar functions defined on $\mathbb{S}$.
Hence, a Kupradze radiating solution has the asymptotic behavior
\ben
u^{sc}(x)=\frac{e^{ik_pr}}{ \sqrt{r}}
u^{\infty}_p(\hat{x})\,\hat{x}+
\frac{e^{ik_sr}}{ \sqrt{r}}
u^{\infty}_s(\hat{x})\,\hat{x}^\perp+O(\frac{1}{r^{3/2}})\quad\mbox{as}\quad r\rightarrow\infty.
\enn
The far-field pattern $u^\infty$ of $u^{sc}$ is defined as
\ben
u^\infty(\hat{x}):=u^{\infty}_p(\hat{x})\,\hat{x}+u^{\infty}_s(\hat{x})\,\hat{x}^\perp.
\enn
Obviously, the compressional and shear parts of the far-field are uniquely determined by $u^\infty$ as follows:
\ben
 u^{\infty}_p(\hat{x})=u^\infty(\hat{x})\cdot \hat{x},\quad u^{\infty}_s(\hat{x})=u^\infty(\hat{x})\cdot \hat{x}^\perp.
 \enn

The first part of this paper is concerned with a uniqueness result within the class of polygonal obstacles defined as follows.
  \begin{definition}
A scatterer $D\subset \R^2$ is called a polygonal obstacle if $D$ is a bounded open set whose boundary $\partial D$ consists of a finite union of line segments and whose closure
$\overline{D}$ coincides with the closure of its interior.
\end{definition}
Throughout this paper we suppose that $D\subset \R^2$ is a connected polygonal obstacle. By the above definition, $D$ consists of a single polygonal domain, and $\partial D$ cannot contain cracks. {\color{rot}{By the elliptic boundary and interior regularity (see \cite{GT, NP1994}), the unique forward solution is $C^{0,\alpha}$-continuous up to the boundary $\partial D$ and belongs to $C^{2,\alpha}(\R^2\backslash\overline{D})\cap H_{loc}^{1+\epsilon}(\R^2\backslash\overline{D})$ for any $\alpha,\epsilon\in(0,1)$.}}
In the following a domain always means a connected open set.
 Our uniqueness result is stated below.
\begin{theorem}\label{TH0}
Assume that $D$ is a connected polygonal obstacle. Then $\partial D$ can be uniquely determined by a single far-field pattern $u^{\infty}(\hat{x})$,
$\hat{x}\in\mathbb{S}$, generated by the incoming plane wave (\ref{planewave}) with fixed incident direction $d\in \mathbb{S}$ and fixed frequency $\omega\in \R^+$.
\end{theorem}

There is a vast literature on inverse elastic scattering problems using the far-field pattern $u^\infty$ {\color{rot}{corresponding to
infinitely many incident directions at a fixed frequency.}} We refer to the first uniqueness result proved in \cite{HH} and
the sampling type inversion algorithms developed in \cite{AK,Arens}.
In these works, not only the pressure part of  far-field patterns for all plane shear and pressure waves  are needed, but also the shear part of far-field patterns.
Uniqueness results using only one type of elastic waves were proved in \cite{GS} and \cite{HKS13}.
It was shown in \cite{HKS13}  that a rigid ball and a convex polyhedron can be uniquely identified by the shear part of the far-field pattern corresponding to only one incident shear wave.

The first global uniqueness results within non-convex polyhedral scatterers
were verified in \cite{EY10}
 with at most two incident plane waves. This extended the acoustic uniqueness results \cite{Rondi05,CY,EY06,EY08, LHZ06} to the third and fourth boundary value  problems of the Navier equation. However, the approach of
\cite{EY10} does not apply to the more practical case of  the first and second kind boundary conditions in elasticity, due to the lack of a corresponding reflection principle for the Navier equation.

The first aim of this paper is to verify Theorem \ref{TH0} through a non-pointwise reflection principle for the Navier equation under the boundary condition of the first kind.
To the best of our knowledge, such a reflection principle is not available in the literature and has been open for a long time.
 The derivation of the reflection principle is based on a revisited Duffin's extension formula (see \cite{Duffin}) for the Lam\'e equation across a straight line;
see Section \ref{sec:2}. The proof of Theorem \ref{TH0} will be presented in Section \ref{sec:3} using a modified path argument. The original path argument employed and developed in \cite{Rondi05, LHZ06, EY06} applies only to boundary conditions with a corresponding reflection principle of "point-to-point" type, and does not extend to the first boundary value problem in linear elasticity.
This paper provides a new approach to prove global uniqueness results within polygonal and polyhedral scatterers in acoustics  and linear elasticity (\cite{Rondi05,CY,EY06,EY10}). In three dimensions, the reflection principles for the Lam\'e and Navier equations can be derived analogously, and the corresponding uniqueness result with a single incoming wave remains valid as well. Our second aim is to propose a revisited factorization method for imaging a rigid elastic body with a single plane wave; see Section \ref{Sec:FM} for the details and additional remarks. The arguments in the proof of Theorem \ref{TH0} will be used to interpret the behavior of our indicator function for polygonal obstacles.

\section{Reflection principles}\label{sec:2}
Throughout this section we denote by $R=R_\Gamma$ the reflection with respect to the straight line $\Gamma:=\{x_1=0\}$, that is, $Rx=(-x_1, x_2)$ for $x=(x_1, x_2)\in \R^2$.
We suppose that
$\gamma\subset \Gamma$ is an open (finite or infinite) line segment lying on $\Gamma$. Let $\Om \subset \R^2$ be a symmetric domain with respect to $\Gamma$ (i.e., $Rx \in \Om$ if
 $x \in \Om$) such that $\gamma$ is a connected component of $\Om \cap \Gamma$.
It is well known that the reflection principle of Schwarz provides a non-local extension
(analytic continuation) formula for a harmonic function
vanishing on a planar boundary surface of the region.
In the following Subsection \ref{subsec:2.1}, we state a relation between the extension formula (reflection principle) and Green's function in a half-plane
for general elliptic operators. Corollary \ref{Coro0} below allows us to construct the half-plane Green's function in terms of the free-plane fundamental solution. The reflection principles for the Lam\'e and Navier equations will be investigated in Subsections \ref{subsec:2.2} and \ref{subsec:2.3}, respectively.

\subsection{Extension formula and Green's function in a half-plane}\label{subsec:2.1}
{\color{rot}{Let $\mathcal{A}$ be a linear elliptic differential operator of second order
with constant coefficients in the symmetric domain $\Om$, and let $\mathcal{B}$ be a first order boundary differential operator with constant coefficients on $\gamma$.
Consider (weak) solutions of the equation $\mathcal{A}v=0$ in $\Om$, which are (real-) analytic in $\Om$ by interior analytic regularity (see e.g., \cite[Chapter 6.4]{GT}).}}
\begin{theorem}\label{TH1}
{\color{rot}{Assume there exists a linear operator $\mathcal{D}$ mapping the space of analytic functions in $\Om$ into itself and such that,
for any solution of $\mathcal{A}v=0$ in $\Om$, the  boundary condition $\mathcal{B}v=0$ holds on $\gamma$ if and only if
\be\label{Extension}
 v(Rx)=\mathcal{D} v(x)\qquad\mbox{for all}\quad x\in \Om.
 \en
}}
Then we obtain:
\begin{itemize}
\item[(i)] If $\mathcal{A}u=0$ in $\Omega$, then the function $w(x):=u(x)+\mathcal{D} u(Rx)$ satisfies the same equation in $\Omega$ and the boundary condition $\mathcal{B}w=0$ on $\gamma$.
\item[(ii)] Denote by $G(x;y) ( x\neq y)$  the half-plane Green's function to $\mathcal{A} $ subject to the boundary condition $\mathcal{B}_x G(x,y)=0$ on $\Gamma$. Then we have the relation $G(Rx,y)=\mathcal{D}_x G(x,y)$ for all $x\neq y$. Here we write $\mathcal{D}=\mathcal{D}_x$ and $\mathcal{B}=\mathcal{B}_x$ to indicate the action
of the differential operators  with respect to the variables $x$.
\end{itemize}
\end{theorem}
\begin{proof} (i)  By (\ref{Extension}), we observe that $\mathcal{D} v(Rx)=v(x)$ and $\mathcal{D}^2v=\mathcal{D}\mathcal{D}v=v$. Hence, $w_1:=\mathcal{D} u(R\cdot)$ satisfies the equation $\mathcal{A}w_1=0$ in $\Omega$, implying that $w=u+w_1$ fulfills the same equation. To prove that $\mathcal{B}w=0$ on $\gamma$, we only need to show that $w(Rx)=\mathcal{D}w(x)$ by our assumption. This follows from the fact that
\ben
w(Rx)=u(Rx)+\mathcal{D}u(x),\quad \mathcal{D}w(x)=\mathcal{D}u(x)+\mathcal{D}^2u(Rx)=\mathcal{D}u(x)+u(Rx).
\enn
(ii) The relation $G(Rx,y)=\mathcal{D}_x G(x,y)$ for all $x\neq y$ simply follows from
(\ref{Extension}).  Note that this relation also implies the singularity of $G(x,y)$ at $x=Ry$.
\end{proof}
Applying Green's formula, one can prove that any function $v$ with $\mathcal{A}v=0$ in $\Omega$, $\mathcal{B}v=0$ on $\gamma$ satisfies the relation (\ref{Extension}) if the half-plane Green's function fulfills this relation. The first assertion of Theorem \ref{TH1} enables us to construct the half-plane Green's function through the free-plane fundamental solution and the extension formula (\ref{Extension}); see Corollary \ref{Coro0} below.
\begin{corollary}\label{Coro0}
Let $\Phi(x,y)$ be the free-plane fundamental solution associated with the operator $\mathcal{A}$, that is,
\ben
\mathcal{A}_x\,\Phi(x,y)=\delta(x-y),\, \qquad x,y\in \R^2, x\neq y.
\enn
Then the function $G(x,y):=\Phi(x,y)+\mathcal{D}_x\, \Phi(Rx,y)$ ($x\neq y$) is  the
half-space Green's function subject to the boundary condition $\mathcal{B}_x G(x,y)=0$ on $\Gamma$.
\end{corollary}

Below we give examples of extension formulas for the first, second and third boundary value problems of the Helmholtz equation in two dimensions. Consider the elliptic operator
$\mathcal{A}=\Delta+k^2$, $k>0$, and the equation $\mathcal{A}v = 0$ in $\R^2$. Then we have the following special cases of Corollary \ref{Coro0}.
\begin{itemize}
{\color{rot}{\item[(a)] Under the Dirichlet boundary condition $\mathcal{B}v:=v=0$ on $\Gamma$, the operator $\mathcal{D}$ can be defined as $\mathcal{D}v=-v$. Note that $v(Rx)=-v(x)$ by the
reflection principle. Moreover, we have $G(x,y)=\Phi(x,y)-\Phi(Rx,y)$ and $G(Rx,y)=\mathcal{D}_x G(x,y)$,
where $\Phi(x,y):=i/4H_0^{(1)}(k|x-y|)$ with $H_0^{(1)}$ being the Hankel function of the first kind of order zero.
\item[(b)] Under the Neumann boundary condition $\mathcal{B}v:=\partial_\nu v=0$ on $\Gamma$, we have $v(Rx)=v(x)$, so that we can choose
$\mathcal{D}v=v$. Furthermore, we then obtain  $G(x,y)=\Phi(x,y)+\Phi(Rx,y)$ and
$G(Rx,y)=\mathcal{D}_x G(x,y)$.
\item[(c)] In the case of the Robin boundary condition $\mathcal{B}v:=\partial_\nu v+q\, v=0$ on $\Gamma$ for some constant $q\in \C$,
we can choose (see \cite{DL1955} for the corresponding reflection principle)}}
\ben
\mathcal{D}v(x) = v(Rx) =  v(x)+2q\int_0^{x_1} e^{(x_1-t)q}v(t,x_2)dt, \, x \in \R^2.
\enn
Consequently, by Corollary \ref{Coro0}, the Green's function in the half-plane $\{x_1>0\}$ takes the form
\ben
 G(x,y)&=&\Phi(x,y)+\mathcal{D}_x \Phi(Rx,y)\\
 &=&\Phi(x,y)+\Phi(Rx,y)+2q\int_0^{-x_1} e^{-(x_1+t)q}\Phi(t,x_2;y)dt.
 \enn
 Using $\mathcal{D}^2=I$, one can also check that
 \ben
 G(Rx,y)=\Phi(Rx,y)+\mathcal{D}_x(x,y)=\mathcal{D}_x G(x,y).
 \enn
\end{itemize}

{\color{rot}{Theorem \ref{TH1}  and Corollary \ref{Coro0} can immediately be extended to elliptic systems of second order with constant coefficients, and in the next
subsection we will apply Corollary \ref{Coro0} to the Lam\'e system in $\R^2$.}}

\subsection{Reflection principle for Lam\'e equation}\label{subsec:2.2}
In this section we consider the extension formula for the Lam\'e operator
\be\label{L0}
\LL_0u=\mu \Delta u+(\lambda+\mu)\nabla(\nabla \cdot u),\quad u=(u_1, u_2)^T.
\en
Assume that
\be\label{Lame}
\LL_0u=0 \quad\mbox{in}\quad \Om,\qquad u=0\quad\mbox{on}\quad \gamma \subset \Gamma:=\{x_1=0\}
\en
in the symmetric domain $\Om$.
Then it is easy to check that
\be\label{eq1}
\Delta^2 u=0,\quad \Delta\; \divv u=0,\quad \Delta\,\overrightarrow{\curl} u=0\qquad \mbox{in}\quad \Om,
\en
where the two-dimensional vector $\curl$ operator is defined by
(\ref{curl}). {\color{rot}{To apply Theorem \ref{TH1} to the Lam\'e operator (\ref{L0}), it is convenient to look first
for an operator $\widetilde{\mathcal{D}}_0$ such that the relation $Ru(Rx) = \widetilde{\mathcal{D}}_0u(x)$,
$x \in \Om$, holds for any solution $u$ of the boundary value problem (\ref{Lame}).
For this purpose, we define the differential operator $\widetilde{\mathcal{D}}_0$ by
\ben
&&\widetilde{\mathcal{D}}_0 u:=-R u+ W_u=(u_1, -u_2)^T+ W_u \, ,\\
&&W_u(x):=c x_1^2 \Delta u(x)+2c x_1(
\partial_2 u_2, -\partial_2 u_1)^T(x),\qquad c:=\frac{\lambda+\mu}{\lambda+3\mu}\, .
\enn
}}
We can prove the following result.
\begin{theorem}\label{TH2}
If $u$ is a solution to (\ref{Lame}), then $\widetilde{\mathcal{D}}_0 u$ is also a solution to (\ref{Lame}), and the relation
$\widetilde{\mathcal{D}}_0 u(x)=Ru(Rx)$ holds for all $x\in \Omega$.
\end{theorem}
\begin{proof}
For notational convenience, we write $u^R(x)=Ru(Rx)$ and $W_u(x)=c x_1 \widetilde{W}_u(x)$, where
\ben
\widetilde{W}_u(x):= x_1 \Delta u(x)+ 2(\partial_2 u_2, -\partial_2 u_1)^T(x).
\enn
{\color{rot}{Since $u$ vanishes on $\gamma \subset \{x_1=0\}$,}} we obtain $\widetilde{W}_u=0$ on $\gamma$. Using (\ref{eq1})  we deduce that
\ben
\Delta\widetilde{W}_u(x)=x_1\Delta^2 u(x)+2\begin{pmatrix}
\Delta \divv u \\ -\Delta\, \overrightarrow{\curl} u
\end{pmatrix}(x)=0.
\enn  Applying the Schwarz reflection principle for harmonic functions gives $\widetilde{W}_u(Rx)=-\widetilde{W}_u(x)$, or equivalently, $\widetilde{W}_u$ is odd in $x_1$.
Moreover, we find
\ben
&&\Delta^2 W_u(x)=c\Delta^2(x_1\widetilde{W}_u(x))=c[x_1\Delta^2 \widetilde{W}_u(x)+4\partial_1 \Delta \widetilde{W}_u(x)]=0,\\
&&\Delta^2 \widetilde{\mathcal{D}}_0 u=-R \Delta^2 u+\Delta^2 W_u=0.
\enn
 Recalling the reflection principle for biharmonic functions with homogeneous Dirichlet data (see \rot{\cite{F1994, Duffin1}}),
we obtain the relation $W_u(Rx)=W_u(x)$, implying that $W_u$ is even is $x_1$. Consequently,  there holds
 \be\label{eq2}
 \partial_1^j W_u(x)=0\qquad\mbox{on}\quad \gamma,\qquad j=0,1,3.
 \en
 To proceed with the proof, we only need to verify that
 \be\label{eqj}
  \partial_1^j \widetilde{D}_0u(x)=\partial_1^j u^R(x)\qquad\mbox{on}\quad \gamma,\qquad j=0,1,2,3.
   \en
 Then the relation $\widetilde{\mathcal{D}}_0u=u^R$ follows from the fact that $\Delta^2 (u^R)=0$, together with the Cauchy-Kovalevskaya theorem. Since the function $V_u:=u^R+Ru$ vanishes on $\gamma$ and is even in $x_1$, it also satisfies the relations in (\ref{eq2}). Hence, it is sufficient to prove (\ref{eqj}) with $j=2$, that is,  $\partial_1^2 V_u=\partial_1^2 W_u$ on $\gamma$.

From the definition of the reflection $R$, it follows that
 \ben
 \partial_1^2 V_u(0, x_2)=2\begin{pmatrix}
 -\partial_1^2 u_1 \\ \partial_1^2 u_2
 \end{pmatrix}(0, x_2).
 \enn
 On the other hand, by the definition of $W_u$,
 \be\label{eq3}
 \partial_1^2 W_u(0, x_2)=2c\partial_1\widetilde{W}_u(0, x_2)=\frac{\lambda+\mu}{\lambda+3\mu} \begin{pmatrix}
 2\Delta u_1+4\partial_1\partial_2 u_2 \\ 2\Delta u_2-4\partial_1\partial_2 u_1
 \end{pmatrix}(0, x_2).
 \en
 Since $\LL_0 u=0$ in $\Omega$ and $\partial_2^2 u_1=\partial_2^2 u_2=0$ on $\gamma$, we obtain
 \ben
 (\lambda+\mu)\partial_1\partial_2 u_2=-(\lambda+2\mu)\partial_1^2 u_2,\quad (\lambda+\mu)\partial_1\partial_2 u_1=-\mu \partial_1^2 u_2.
 \enn
 Therefore, it follows from (\ref{eq3}) that
 \ben
 \partial_1^2 W_u=\frac{1}{\lambda+3\mu}\begin{pmatrix}
 [2(\lambda+u)-4(\lambda+2\mu)] \partial_1^2 u_1\\
 [2(\lambda+\mu)+4\mu] \partial_1^2 u_2
 \end{pmatrix}=2\begin{pmatrix}
 -\partial_1^2 u_1 \\ \partial_1^2 u_2
 \end{pmatrix}\quad\mbox{on}\quad\gamma,
 \enn
 which proves $\partial_1^2V_u=\partial_1^2W_u$ on $\gamma$.
 \end{proof}

Let $\Omega^+ \subset \{x_1>0\}$ be a domain such that $\gamma \subset \partial \Omega^+$, and let
$\Omega^- := R(\Omega^+)$. {\color{rot}{Then $\Omega := \Omega^+ \cup \gamma \cup \Omega^-$ is a symmetric domain with $\Omega \cap \Gamma = \gamma$, and
from Theorem \ref{TH2} we obtain an extension formula for the first boundary value problem of the Lam\'e equation across a  straight line:}}
\begin{corollary}\label{Coro1} Suppose that $\LL_0 u=0$ in $\Omega^+$, $u=0$ on $\gamma$. Define the function
\ben
u^*(x)=\left\{\begin{array}{lll}
u(x)\quad&&\mbox{if}\quad x\in \Om^+,\\
\mathcal{D}_0u(Rx)\quad&&\mbox{if}\quad x\in \Om^-,
\end{array}\right.
\enn
where $\mathcal{D}_0 u:=R\widetilde{\mathcal{D}}_0u$ takes the explicit form
\be\label{D0}
\mathcal{D}_0 u=-u+c x_1^2\begin{pmatrix}
-\Delta u_1 \\ \Delta u_2
\end{pmatrix} -2c x_1 \begin{pmatrix} \partial_2 u_2 \\ \partial_2 u_1
\end{pmatrix}.
\en Then $u^*$ is a solution to (\ref{Lame}).
\end{corollary}
{\color{rot}{The relation $u(x) = \mathcal{D}_0u(Rx) \, , \; x \in \Om^-$,
coincides with Duffin's extension formula \cite[Theorem 2]{Duffin}. }}
It follows from the above corollary that $u^*(Rx)=\mathcal{D}_0u^*(x)$ for all $x\in \Om$. By the definition of $\mathcal{D}_0$, we conclude that the value of $u^*$ at $Rx$ is uniquely determined by $u$ in a neighborhood of the imaging point $x$, which is in contrast to the point-to-point reflection principles for the Laplace and Helmholtz equations under the Dirichlet or Neumann boundary condition.
Combining Corollaries \ref{Coro0} and \ref{Coro1}, we can obtain the Green's tensor to the first boundary value problem of the Lam\'e equation in the half-plane $\{x_1>0\}$, that is,
\be\label{G0}
G_0(x,y)=\Phi_0(x,y)+\mathcal{D}_0\Phi_0(Rx, y)
\en where $\mathcal{D}_0$ is defined via (\ref{D0}) and $\Phi_0(x,y)$ is the free-plane Green's tensor to the Lam\'e operator, given by (see \cite[Chapter 2.2]{HW2008})
\ben
\Phi_0(x,y)=\frac{1}{4\pi}\left[ -\frac{3\mu+\lambda}{\mu (\lambda+2\mu)}\,\ln |x-y|\,\textbf{I}+\frac{\lambda+\mu}{\mu (\lambda+2\mu)|x-y|^2} (x-y)\otimes (x-y)\right],\quad x\neq y.
\enn
Here $\textbf{I}\in \R^{2\times 2}$ is the 2-by-2 identity matrix and $(x\otimes y)_{ij}:=x_iy_j$ for $i,j=1,2$, where $x=(x_1, x_2)$, $y=(y_1, y_2)\in \R^2.$ The extension formula and
Green's tensor in the half-plane $\{x_2>0\}$ can be obtained analogously by a coordinate rotation.

\subsection{Reflection principle for Navier equation}\label{subsec:2.3}
Consider the boundary value problem
\be\label{Navier}
\LL_\omega u:=(\LL_0+\omega^2 )u=0 \quad\mbox{in}\quad\Om,\qquad u=0\quad\mbox{on}\quad\gamma\subset \{x_1=0\}
\en
for the Navier equation in the symmetric domain $\Om$.
We want to find a formula connecting $u(Rx)$ and $u(x)$ for all $x\in\Om$. Our approach relies on the extension formula for the Lam\'e operator presented in Corollary \ref{Coro1}.

Let $G_0(x,y)$ be the half-plane Green's tensor to the first boundary value problem of the Lam\'e equation; see (\ref{G0}).  For $\delta>0$ sufficiently small, define $\Om_\delta:=\{x\in\Omega: \mbox{dist}(x, \partial \Om)>\delta\}$.  Introduce the function
\ben
v(x)=u(x)-\omega^2\int_{B_\delta(x)} G_0(x, y)^T u(y)\,dy,\quad x\in\Om_\delta,
\enn
where $B_\delta(x)=\{y: |y-x|<\delta\}$. Then it is easy to check that $v$ fulfills the homogeneous Lam\'e equation with the homogeneous Dirichlet boundary condition, that is,
\ben
\LL_0 v=0\quad\mbox{in}\quad\Om_\delta,\qquad v=0\quad\mbox{on}\quad \gamma\cap \Om_\delta.
\enn
Applying Theorem \ref{TH2} and the definition of $\mathcal{D}_0$ in (\ref{D0}) to $v$, we obtain $v(Rx)=\mathcal{D}_0 v(x)$ for all $x\in \Omega_\delta$, that is,
\be\nonumber
u(Rx)&=&\mathcal{D}_0u(x)-\omega^2\int_{B_\delta(x)} {\color{rot}{\mathcal{D}_0 \,G_0(x, y)^T u(y)}}\,dy+\omega^2\int_{B_\delta(x)} G_0(Rx, y)^T u(y)\,dy\\ \label{Ex-Navier}
&=:& \mathcal{D}_\omega u(x).
\en
The above equality establishes a relation between $u(Rx)$ and $u(x)$. For every fixed $x\in \Om$, the number $\delta$ on  the right hand side of (\ref{Ex-Navier}) can be replaced by any number less that $\mbox{dist}(x, \partial\Om)$. In fact, for any $\epsilon\in(0,\delta)$, it holds that (see Theorem \ref{TH1} (ii))
\ben
G_0(Rx,y)=\mathcal{D}_0 G_0(x,y)\quad \mbox{for all}\quad y \in B_\delta(x)\backslash\overline{B}_\epsilon(x),\enn
from which the relation
\ben
-\omega^2\int_{B_\delta(x)} {\color{rot}{\mathcal{D}_0 \,G_0(x, y)^T u(y)}}\,dy+\omega^2\int_{B_\delta(x)} G_0(Rx, y)^T u(y)\,dy\\
=-\omega^2\int_{B_\epsilon(x)} {\color{rot}{\mathcal{D}_0 \,G_0(x, y)^T u(y)}}\,dy+\omega^2\int_{B_\epsilon(x)} G_0(Rx, y)^T u(y)\,dy
\enn
follows.
Hence, the function $\mathcal{D}_\omega u(x)$ on the right hand side of (\ref{Ex-Navier}) is well defined as long as $u$ makes sense in a neighboring area of $x\in \Om$.
Moreover, we observe that as for the Lam\'e equation, the value of $u(Rx)$ is uniquely determined by the function $u$ near the imaging point $x$.  Note that  the volume $B_\delta(x)$ on the right hand side of (\ref{Ex-Navier}) can also be replaced by any domain containing $x$, for example, the region $\Omega$ (provided it is bounded).
The reflection principle for the Navier equation will be summarized in the following theorem.

\begin{theorem}\label{TH3}
Let $u$ be a solution to (\ref{Navier}). Then
\begin{itemize}
\item[(i)] It holds that $u(Rx)=\mathcal{D}_\omega u(x)$ for all $x\in \Omega_\delta$.
\item[(ii)] The function $w(x):=R[\mathcal{D}_\omega u](x)$ satisfies
\ben
\LL_\omega w=0\quad\mbox{in}\quad\Om_\delta,\qquad w=0\quad\mbox{on}\quad \gamma\cap \Om_\delta.\enn
Further, we have the relation $w(x)=Ru(Rx)$ for all $x\in \Omega_\delta$.
\item[(iii)] The function $w$ defined in assertion (ii) can be extended into the whole domain $\Omega$ as a solution of the Navier equation.
In particular, we have $w=0$ on  a smooth part $\gamma_1 \subset \partial\Om$ if $u=0$ on $R(\gamma_1)$.
\end{itemize}
\end{theorem}
In the application of the reflection formula (\ref{Ex-Navier}) to inverse elastic scattering, we need a corresponding analytic extension result.
Let $D^+\subset \R^2$ and $\Pi\subset \R^2\backslash\overline{D^+}$ be domains with piecewise smooth boundary (in particular, polygonal domains) and suppose that
$\gamma \subset \partial D^+ \cap \partial \Pi$. Then $\Omega := D^+\cup\gamma\cup\Pi$ is also a domain with
piecewise smooth boundary. Moreover, as in (\ref{Ex-Navier}) we define the function $\mathcal{D}_\omega u(x), \, x \in \Omega_\delta$.

\begin{lemma}\label{re-rp}
Consider the boundary value problem
\ben
\left\{\begin{array}{lll}
\LL_\omega u=0 \quad\mbox{in} \quad \Omega \, ,\\
 \quad u=0 \quad\mbox{on}\quad \partial D^+.
 \end{array}\right.
\enn
Then we have:
\begin{itemize}
\item[(i)]
The function $w(x):=R[\mathcal{D}_\omega u](x)$, $x\in \Omega_\delta$, can be analytically extended into $D^-:=R(D^+)$
as a solution of the Navier equation. Moreover, the extended solution satisfies the relations
\ben
w(x)=Ru(Rx)\quad\mbox{in}\quad D^-,\quad w=0\quad\mbox{on}\quad \partial D^-.
\enn
\item[(ii)] 
Suppose that $\Omega$ contains a half-plane whose boundary is the extension of
one segment of  $\partial D^-$ in $\R^2$.
 Then both $w$ and $u$ can be analytically extended onto the whole plane $\R^2$.
\end{itemize}
\end{lemma}
\begin{proof} (i)
By the interior regularity for elliptic equations, $u$ is analytic in $\Om$ and thus $w$ is analytic in $\Omega_\delta$. In view of Theorem \ref{TH3}, we first have the coincidence
$w(x)=Ru(Rx)$ in the connected component of $\Omega_\delta\cap D^-$ containing $\gamma\cap \Omega_\delta$, and both functions
 fulfill the Navier equation there. On the other hand, the function $x\rightarrow Ru(Rx)$ obviously fulfills the Navier equation in $D^-$. This implies that $w$ can be analytically extended into $D^-$ by
$Ru(R\cdot)$, and in particular $w=0$ on $\partial D^-$, since $u=0$ on $\partial D^+$.

ii) 
Assume that $l\in \partial D^-$ is a line segment lying on the straight line $L=\{x: x_2=a x_1+b,  x_1\in \R\}$ for some $a, b\in \R$ and that the half-plane
$\{x: x_2> a x_1+b, x_1\in \R\}$ is contained in $\Omega$. Let $w$ be the function defined in the first assertion. Then we have $w=0$ on $l$ and, by the analyticity of $w$ in $\Omega$, $w=0$ on $L$.  
Applying coordinate translation and rotation, we assume that the orthogonal matrix $Q$ transforms $L$ to the line $\{x_1=0\}$ and transforms the above mentioned half-plane
inside $\Omega$ to $\{x_1> 0\}$. Since the Navier equation remains invariant under the transformation $Q$, the function  $\tilde{w}(x):=w(Qx)$ satisfies
\ben
\LL_\omega \tilde{w}=0\quad\mbox{in}\quad x_1>0,\qquad \tilde{w}=0 \quad\mbox{on}\quad x_1=0.
\enn
By Theorem \ref{TH3},  $\tilde{w}$ can be analytically extended into $\R^2$ by $\tilde{w}(x)=\mathcal{D}_\omega \tilde{w}(Rx)$ for $x_1<0$.
This in turn implies that $w$ and thus $u$ can be analytically extended onto the whole plane $\R^2$.

\end{proof}

\section{Uniqueness to inverse elastic scattering}\label{sec:3}
Consider elastic scattering from a rigid obstacle $D\subset\R^2$ modeled by
\ben\left\{\begin{array}{lll}
\LL_\omega u=0\quad\quad \mbox{in}\quad D^c,\qquad u=u^{in}+u^{sc},\\
u=0\quad\qquad\mbox{on}\quad \partial D,\\
u^{in}(x)=c_p d e^{ik_p x\cdot d}+c_sd^\perp e^{ik_s x\cdot d},\\
\mbox{$u^{sc}$ satisfies the Kupradze radiation condition stated in Definition \ref{Def}.}
\end{array}\right.
\enn
To prove the uniqueness result with a single plane wave, we need the concept of nodal set of a solution $u$ to the above boundary value problem.
\begin{definition}
The nodal set $\mathcal{N}$ of $u$ consists of all points $x\in \R^2\backslash\overline{D}$ such that $u(x)=0$.
\end{definition}
The reflection principle for the Navier equation (Theorem \ref{TH3} and Lemma \ref{re-rp})  can be used to prove the following lemma.

\begin{lemma}\label{TH4}
If $D$ is a connected polygonal obstacle,
then the nodal set $\mathcal{N}$ of $u$ cannot contain a line segment with both end points lying on $\partial D$.
\end{lemma}
\begin{proof} Assume contrarily that $l_0\subset \mathcal{N}$ is a line segment with the two end points on $\partial D$.
Choose a point $P_0\in l_0$ and a continuous and injective path $\pi(t)$, $t\geq0$, starting at $P_0=\pi(0)$ and leading to infinity in the unbounded component of $D^c\backslash l_0$. Denote by $E_0\subset D^c$ the bounded component of $D^c\backslash l_0$; recall that $D$ and $D^c$ are polygonal domains without cracks, and $D$ is bounded.


Then $\partial E_0\subset \partial D\cup l_0$ and we have
\ben
\LL_\omega u=0\quad\mbox{in}\quad E_0,\qquad u=0\quad\mbox{on}\quad \partial E_0.
\enn
In what follows we denote by $R_n$ ($n\geq 0$) the reflection with respect to the straight line $L_n$ containing the line segment $l_n$, and by $R_n^o$
the reflection with respect to the straight line $L_n^o$ that is parallel to $L_n$ and contains the origin $(0,0) \in \R^2$. Moreover, let
$\mathcal{D}_\omega^{(n)}$ denote the reflection operator for the Navier equation with respect to the line segment $l_n \subset L_n $, which is obtained as in
(\ref{Ex-Navier}) after translation and rotation.
Write $E_1=R_0(E_0)$. Obviously, the function $u_0:=u|_{E_0}$ can be analytically extended into $D^c\backslash E_0$ across the line segment $l_0$, and in particular, $u_0=u$
is well-defined near the path $\pi(t)$ in $E_1$.

Transforming $L_0 $ to the line $\{x_1=0\}$ by translation and rotation, from
Lemma  \ref{re-rp} (i) we obtain that the function $u_1(x):=R_0^o[\mathcal{D}^{(0)}_\omega u](x)$, $x \in D^c$, satisfies the relation
\ben
u_1(x)=R_0^ou(R_0x), \; x \in E_1,
\enn
and the boundary value problem
\ben
\LL_\omega u_1=0\quad\mbox{in}\quad E_1,\qquad u_1=0\quad\mbox{on}\quad \partial E_1.
\enn
Since $E_1$ is bounded, we see that $\partial E_1\cap \{\pi(t): t>0\}\neq\emptyset$. Set $t_1:=\sup\{t: E_1\cap \pi(t)\neq\emptyset\}$. Without loss of generality we suppose that $P_1:=\pi(t_1)$
is not a corner point of $\partial E_1$. Note that this can always be achieved by locally changing the path near $t=t_1$ if necessary.
Hence it holds that $P_1\neq P_0$ and $E_1\cap\{\pi(t): t>t_1\}=\emptyset$.
Denote by $l_1\subset\partial E_1$
the line segment containing the point $P_1$.  Setting $E_2=R_1(E_1)$ and
applying Lemma \ref{re-rp} again, we can repeat the previous step to define a function $u_2$ defined in $D^c$, which satisfies the relation
\ben
u_2(x):=R_1^o[\mathcal{D}^{(1)}_\omega u_1](x)=R_1[u_1(R_1x)], \quad x\in E_2
\enn
and the Navier equation in $E_2$ with vanishing Dirichlet data on $\partial E_2$. Moreover, we can find a point $P_2:=\pi(t_2)\neq P_1$ for some $t_2>t_1$ and
a line segment $l_2\subset \partial E_2$ such that $P_2\in l_2$ and $E_2\cap\{\pi(t): t>t_2\}=\emptyset$.

After a finite number of steps, we find a polygonal domain $E_N$, $N \geq 1$, and a function
\be \label{func}
u_N(x):=[R_{N-1}^o \mathcal{D}_\omega^{(N-1)} \, R_{N-2}^o \mathcal{D}_\omega^{(N-2)} \, ... \, R_{0}^o \mathcal{D}_\omega^{o}u](x), \; x \in D^c
\en
such that
\ben
u_N(x):=\tilde{R}_N^ou(\tilde{R}_Nx),\quad x\in E_N,\quad \tilde{R}_N^o:= R^o_{N-1} R^o_{N-2} ... R^o_{0}, \quad \tilde{R}_N:= R_{0} R_{1} ... R_{N-1}
\enn
and
\ben
\LL_\omega u_N=0\quad\mbox{in}\quad E_N,\qquad u_N=0\quad\mbox{on}\quad \partial E_N.
\enn
Moreover, we may suppose $\{\pi(t): t>t_{N}\}\cap E_N=\emptyset$ for some $t_{N}>t_{N-1}$ and that $P_{N}=\pi(t_{N})\in l_{N}$, where $l_{N}\subset \partial E_N$ is a line segment.
Since the path $\pi(t)$ is connected to infinity in $D^c$, by Lemma \ref{infty} below we can assume that $\mbox{dist}(E_{N}, \partial D)\gg 1$. Moreover, we can suppose
that there is a line segment $l\subset \partial E_N$ whose maximal extension $L$ in $D^c$ does not intersect $\partial D$. This follows from the fact that $\partial E_{N}$ always contains at least
two segments forming a positive angle $\leq \pi/2$ which is bounded from zero uniformly in $N$.

The property of $L$ together with the relation (\ref{func}) implies that the functions $u_N$ and $u_{N-1}$ are well defined in an unbounded domain containing the half-plane with
the boundary $L_N$, being the extension of $l_N\subset E_N=R_N(E_{N-1})$ in $D^c$.
%
Now, applying
Lemma \ref{re-rp} enables us to extend $u_N$ and $u_{N-1}$ into the whole plane as a solution of the Navier equation (set $u=u_{N-1}$, $w=u_N$ and $D^+=E_{N-1}$, $D^-=E_N$ and $\gamma=l_{N-1}$ in Lemma \ref{re-rp}).
The analytical extension of $u_{N-1}$ in turn implies that $u_n$ ($0<n\leq N-2$) and, in particular, $u_0=u|_{E_0}$ can be extended onto $\R^2$ as well.  In fact, this can be proved in the
same manner as in the proof of Lemma \ref{re-rp} (ii).

Hence, the scattered field $u^{sc}$ can be extended onto $\overline{D}$ as an entire Kupradze radiation solution, implying that $u^{sc}\equiv0$ in $\R^2$. {\color{rot}{
This implies $u^{in}=0$ on $\partial D$ due to the boundary condition of the total field $u$. Hence,
\ben
|c_p|+|c_s|=|d\cdot u^{in}(x)|+|d^\perp\cdot u^{in}(x)|=0,\qquad x\in \partial D,
\enn
which contradicts the assumption that $|c_p|+|c_s|\neq 0$.}}
%
\end{proof}
In the proof of Lemma \ref{TH4} we need the following result.
\begin{lemma}\label{infty}
Let $P_n=\pi(t_n)\in l_n\cap D^c$ with $t_{n+1}>t_n\, (n=0, 1,\cdots)$  be the points constructed in the proof of Lemma \ref{TH4}. We have
\ben
\lim_{n\rightarrow\infty}\mbox{dist}\, (P_n, \partial D)=\infty.
\enn
\end{lemma}
\begin{proof} We keep the notation used in the proof of Lemma \ref{TH4}.
Suppose on the contrary that  $\mbox{dist}(P_n, \partial D)<\infty$ as $n\rightarrow\infty$.  We can always choose a subsequence, which we still denote by $P_n$, such that $P_n\rightarrow P^*$ for some $P^*$ and $t_n\rightarrow t^*$ as $n\rightarrow\infty$, where $t^*<\infty$.  Note that, if $t^*=\infty$, one can see that $\lim_{n\rightarrow\infty}\pi(t_n)<\infty$, contradicting the fact that
$\pi(t)$ is connected to infinity.

Further, we may suppose that there exists $N>0$ such that
\be\label{eq:P}
|P^*P_j|<\epsilon,\quad
|t_j -t^*|<\epsilon, \quad P_j\neq P_{j+1}\qquad\mbox{for all}\quad j\geq N-1.
\en
Since $\epsilon>0$ can be arbitrarily small, this implies that $l_{N-1}$ and $l_N$ must be two neighboring line segments lying on the boundary of the polygonal domain $E_N$. Without loss of generality, we assume that the corner point $l_{N-1} \cap l_N$ coincides with the origin and that
\ben
l_{N-1}\subset\{(r,\theta): \theta=- \phi_0\},\quad
l_N\subset\{(r,\theta):\theta=0\}
\enn
for some $\phi_0\in(0, 2\pi)$, where $(r, \theta)$ denote the polar coordinates.  Let $J=\inf_{j\geq 0}\{j\phi_0\geq 2\pi\}$.
It follows from (\ref{eq:P}) that
\[P_{N+j}=\pi(t_{N+j})\in l_{N+j}\subset \{(r,\theta): \theta=j\phi_0\}
\] for all $1\leq j\leq J$. Then, after the $J$-th reflection we have $l_{N}\cap \overline{E}_{N+J}\neq \emptyset$ and thus there exists $t_{N+J}^*$ such that
$\pi(t_{N+J}^*)=P_{N+J}^*\in\{\pi(t): t\in(t_{N+J-1}, t_{N+J}]\}\cap l_{N}$. By the injectivity of the path $\pi(t)$, it follows that the arc
$\{\pi(t): t\in(t_{N+J-1}, t_{N+J})\}$ cannot intersect the arc of $\pi(t)$
for $t\in (t_{N-1}, t_N)$. Therefore, the point $P_{N+J}^*\in l_N$ must lie between the origin and $P_N\in l_N$, that is,  $ |OP_{N+J}^*|<|OP_N|$. Hence, we obtain
\ben
E_N\cap\{\pi(t): t\in(t_{N+J-1}, t_{N+J}]\}\neq\emptyset,
\enn which contradicts the fact that
\ben
E_N\cap\{\pi(t): t>t_N\}=\emptyset.
\enn
\end{proof}

We remark that if the nodal set contains a line segment whose end points lie on $\partial D$, then a non-uniqueness example to inverse scattering can be easily constructed;
see the example at the end of this section.
Having proved the property of the nodal set in Lemma \ref{TH4}, we can verify the main uniqueness result of Theorem \ref{TH0} with a single incoming plane wave.

\textbf{Proof of Theorem \ref{TH0}}. Suppose there are two rigid polygonal obstacles $D_1$ and $D_2$ such that $u^\infty_1=u^\infty_2$ for the incoming plane wave (\ref{planewave}). Here and thereafter we denote by $u^\infty_j, u_j, u_j^{sc}$ ($j=1,2$) the far-field patterns, and the total and scattered fields corresponding to $D_j$.
By Rellich's lemma, we obtain $u_1=u_2$ in the unbounded component $E$ of $\R^2\backslash\overline{D_1\cup D_2}$.

If $D_1\neq D_2$, one can always find a finite line segment $l$ such that, without loss of generality, $l\subset \partial D_1\cap \partial E$ but $l\cap \overline{D}_2=\emptyset$. Denote by $L$ the maximum extension of $l$ in $D_2^c:=\R^2\backslash{\overline{D}}_2$. Since $u_2$ is real analytic in $D^c_2$, we get $u_2=0$ on $L$, that is, $L$ is a subset of the nodal set of $u_2$. By Lemma \ref{TH4}, $L$ cannot be a finite line segment with the end points lying on $\partial D_2$.
Hence, $L$ must be connected to infinity in $D^c_2$.  In view of the Kupradze radiation condition of $u^{sc}_2$, we get
\ben
\lim_{|x|\rightarrow\infty,\, x\in L} u_2^{sc}(x)=0,
\enn
which gives rise to the same asymptotic behavior of $u^{in}$ on $L$ and thus
\ben
|d\cdot u^{in}(x)|+|d^\perp\cdot u^{in}(x)|\rightarrow 0,\quad\mbox{as}\quad |x|\rightarrow \infty, \quad x\in L.
\enn
However, the previous relation is impossible, because
\ben
|d\cdot u^{in}(x)|+|d^\perp\cdot u^{in}(x)|=|c_p|+|c_s|\neq 0,\qquad x\in\R^2.
\enn
This contradiction implies that $D_1=D_2$.
$\hfill\Box$

\begin{remark}\label{rem:sing}
From Lemma \ref{TH4} and the proof of Theorem \ref{TH0} we conclude that
\begin{itemize}
\item[(i)] The nodal set $\mathcal{N}$ of $u$ cannot coincide with a finite line segment.
\item[(ii)] The total field $u$ must be singular at each corner point lying on the convex hull of $\partial D$. In other words, $u$ cannot be analytically extended into $D$ across a corner point of
 the convex hull of $\partial D$. This fact will be used in Section \ref{Sec:FM} below to interpret the behavior of an indicator function for imaging rigid polygonal obstacles; see Remark \ref{Re:obstacle}.
\end{itemize}
\end{remark}

For the readers' convenience,  we finally illustrate the idea in the proofs of Theorem \ref{TH0} and Lemma \ref{TH4} through a simple example.
We shall construct two concrete polygonal obstacles and show why they cannot generate identical scattering data. Let $D$ be given as in Figure \ref{F} and let the line segment $\gamma$ be part of the nodal set of $u$ corresponding to $D$ and some fixed incident plane wave. This implies that the polygonal obstacles $D$ and $\tilde{D}:=(D\cup \overline{\Omega}^+)\backslash \gamma$ would generate identical scattering data, where $\Omega^+$ is the gap domain between $D$ and $\tilde{D}$.
By the proof of Lemma \ref{TH4},  the function $u_0:=u|_{\Omega^+}$ is a solution to the Navier equation in $\Omega^+$ with vanishing Dirichlet data on $\partial \Omega^+$ and  $u_1(x):=R \mathcal{D}_\omega u(x)=Ru(Rx)$ satisfies the same boundary value problem over $\Omega^-=R(\Omega^+)$. Since $D^c$ contains the half plane $\{x_1\leq c\}$ for some $c<0$, the function $u_1$ is also well defined
over $\{x_1\leq c\}\cup \Omega^-$ and in particular, $u_1=0$ on $\{x_1=c\}$ \rot{by analyticity}. Applying the reflection principle of the Navier equation, $u_1$ can be analytically extended onto $\R^2$. This implies that $u_0$ and thus $u^{sc}$ can be also extended onto the whole space, which is impossible.
For more general configurations of two polygonal obstacles, the multiple reflection and path arguments presented in the proof of Lemma \ref{TH4} can be used to derive a contradiction.
\begin{figure}[h]
\begin{center}
\includegraphics[width=8cm]{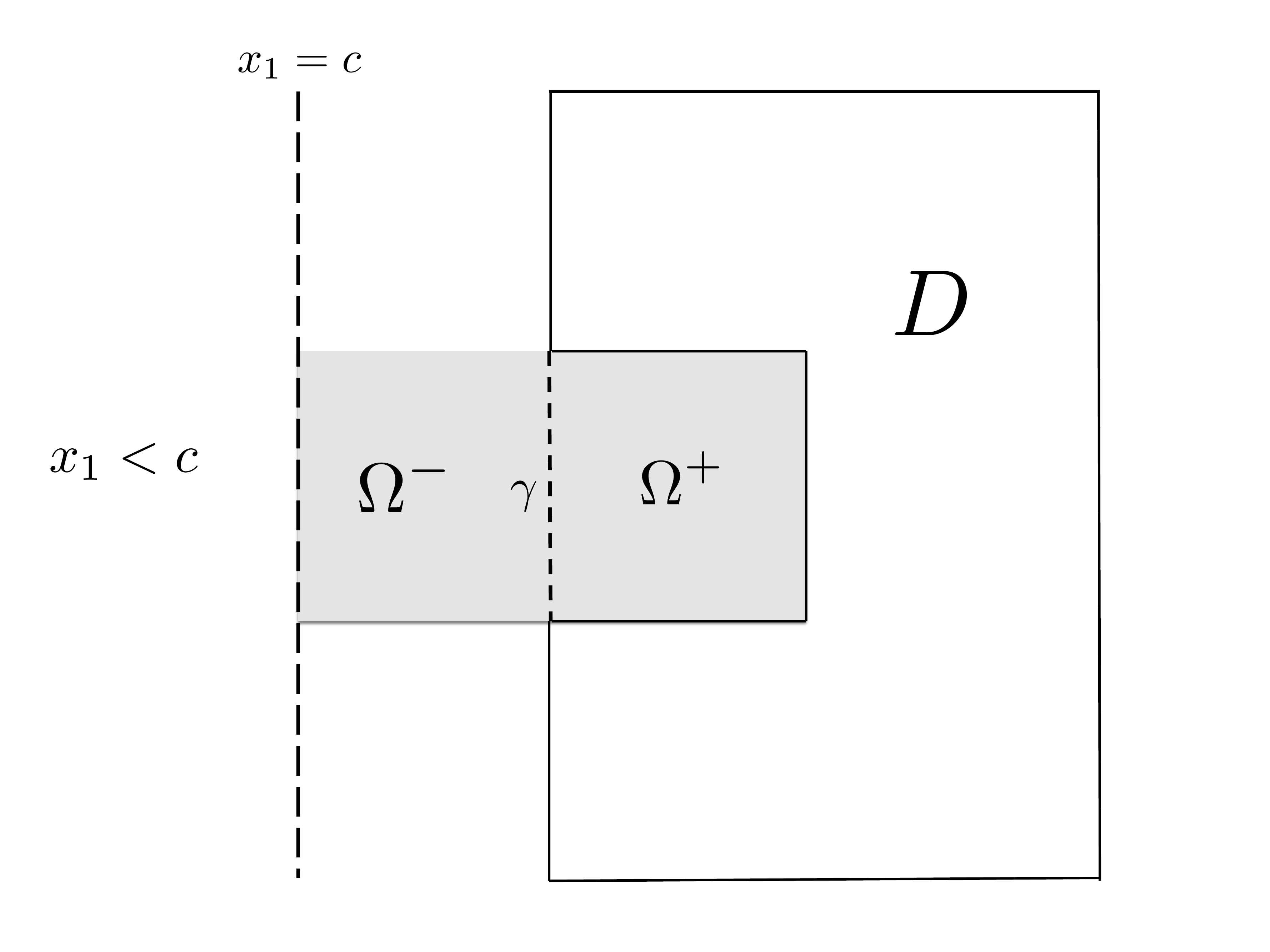}
\caption{\label{F}Illustration of the idea in the proof of Lemma \ref{TH4}: $D$ is polygonal obstacle and $\gamma\subset\{x_1=0\}$ is one line segment of the nodal set with two end points lying on $\partial D$. In this case, a contradiction can be easily deduced by the reflection with respect to $\gamma$.}
\end{center}
\end{figure}

\section{Factorization method with a single far-field pattern}\label{Sec:FM}
The aim of this section is to propose a revisited factorization method for recovering $D$ from a single far-field pattern. The original factorization method \cite{Kirsch98, Kirsch08} by A. Kirsch makes use of knowledge of far-field patterns corresponding to all incident directions. We take inspiration from a recent paper \cite{LS2018}  on an extended linear sampling method with a single plane wave and improve the analysis and the inversion scheme there within the framework of factorization method.  A comparison of our approach to \cite{LS2018}  will be given at the end of this section.
\subsection{Factorization method with infinitely many incoming directions}\label{sub:fac}
We first present a brief review of the factorization method in linear elasticity established in \cite{AK, Arens}. For $g\in L^2(\s)^2$, introduce the Herglotz operator $H: L^2(\s)^2\rightarrow H^{1/2}(\partial D)^2$ by
\be\label{Herglotz}
Hg(x)=\int_{\s}\left[ g_s(d)\,d^\perp e^{ik_s x\cdot d}+g_p(d)\,d\,e^{ik_p x\cdot d}\right]\,ds(d),\qquad x\in \partial D,
\en where $g_s(d):=g(d)\cdot d^\perp$ and $g_p(d):=g(d)\cdot d$ are the tangential and normal components of $g\in L^2(\s)^2$, respectively.
The far-field operator $F: L^2(\s)^2\rightarrow L^2(\s)^2$ is defined by
\ben
Fg(\hat{x})=\int_{\s} \left[g_s(d)\,u^\infty_s(\hat{x}, d)+ g_p(d)\,u^\infty_p(\hat{x}, d)\right]\,ds(d),\quad \hat{x}\in \s,
\enn where $u^\infty_s$ and $u^\infty_p$ are the far-field patterns incited by the incident plane wave $d^\perp \exp(ik_s x\cdot d)$ and $d \exp(ik_p x\cdot d)$, respectively.
The function $Fg(\hat{x})$ is the far-field pattern corresponding to the incident wave defined by the right hand side of (\ref{Herglotz}). It was shown in \cite[Theorem 3.3]{Arens} that $F$ is compact and normal. Denote by $\Phi_\omega$ the Green's tensor of the Navier equation in two dimensions, give by
\ben
\Phi_\omega(x, y)=\frac{1}{4\mu}H_0^{(1)}(k_s|x-y|)+\frac{i}{4\omega} \grad_x\grad_x^\perp \left[ H_0^{(1)}(k_s|x-y|)-H_0^{(1)}(k_p|x-y|) \right].
\enn
The far-field pattern of the function $x\rightarrow \Phi_\omega(x,y)P$ for some fixed polarization vector $P\in \R^2$ is given by
\ben
\Phi^\infty_{y}(\hat{x})=\frac{e^{-ik_p\hat{x}\cdot y+i\pi/4}}{\sqrt{8\pi k_p}}(\hat{x}\cdot P)\,\hat{x}+\frac{e^{-ik_s\hat{x}\cdot y+i\pi/4}}{\sqrt{8\pi k_s}}(\hat{x}^\perp\cdot P)\,\hat{x}^\perp,\quad \hat{x}\in \s. \enn
It was proved in \cite{AK} and \cite{Arens} that the function $\Phi^\infty_y$ can be used to characterize the scatterer $D$ in terms of the range of $(F^*F)^{1/4}$. Using the orthogonal system of eigenfunctions of $F$, Picard's theorem then implies the following result.
\begin{proposition}
If $\omega^2$ is not a Dirichlet eigenvalue of the operator $-\LL_0$ in $D$, then
\be\label{factorization}
y\in D \quad {\color{rot}{\mbox{if and only if} }}\quad   W(y):=\left[ \sum_{n=1}^\infty \frac{|(\phi_n, \Phi_y^\infty)_{L^2(\s)^2}|^2}{|\eta_n|}    \right]^{-1}>0,
\en where $\eta_n\in \C$ denotes the eigenvalues of $F$ with the corresponding eigenfunctions $\phi_n\in L^2(\s)^2$.
\end{proposition}
We refer to \cite{HKS13} for the factorization method using the shear (resp. compressional) part of the far-field pattern corresponding to all incident shear (resp. compressional) plane waves with all directions. We remark that the right hand side of (\ref{factorization}) is the inverse of the $L^2$-norm of the solution $g$ to the operator equation
\ben
(F^*F)^{1/4} g=\Phi_y^\infty.
\enn In fact, the above equation is solvable (that is, $\Phi_y^\infty\in\mbox{Range}(F^*F)^{1/4}$) if and only if $y\in D$, and the unique solution is given by
\ben
g(\hat{x})=\sum_{n=1}^\infty \frac{(\phi_n, \Phi_y^\infty)_{L^2(\s)^2}}{\sqrt{|\eta_n|}} \phi_n(\hat{x}) ,\qquad y\in D.
\enn

\subsection{Factorization method with a single incoming wave}
Assume that the unknown rigid scattered $D$ is contained in $B_R=\{x: |x|<R\}$ and that $D$ is connected. We want to recover $\partial D$ from a single far-field pattern $u^\infty(\hat{x})$ generated by one incoming elastic plane wave of the form (\ref{planewave}) with the fixed direction $d\in \s$ and frequency $\omega\in \R^+$.

Let $z=R(\cos\theta, \sin\theta)\in \Gamma_R=\partial B_R$ and let $B_h(z)=\{x\in \R^2: |x-z|=h\}$ be a disk with radius $h>0$ centered at $z$. For simplicity we write $B_h(z)=B_{h, \theta}$ where $\theta\in[0, 2\pi)$ and $h\in(0, 2R]$ will be called the sampling variables.
Suppose that $B_{h,\theta}$ is a rigid disk, and denote by $F_{h,\theta}$ the far-field operator associated with $B_{h,\theta}$. Consider the operator equation
\be\label{F*}
(F^*_{h,\theta}F_{h,\theta})^{1/4} \,g=u^\infty.
\en
We want to characterize $D$ through the solution $g=g_{h,\theta}$ of the above operator equation for all sampling variables $h$ and $\theta$. To introduce our indicator function, we need to
define the minimum and maximum distance between $z$ and $\partial D$ by
\ben
l_z=\mbox{dist}(z, \partial D)=\min_{x\in \partial D}|x-z|,\qquad L_z:=\max_{z\in\partial D}|x-z|.
\enn
Below we adapt the arguments of Subsection \ref{sub:fac} to the solvability of (\ref{F*}) .
\begin{theorem}
Let $z=R(\cos\theta,\sin\theta)$ and $h\in(0, 2R]$ be fixed, and suppose that $\omega^2$ is not a Dirichlet eigenvalue of the operator $-\LL_0$ in $B_{h,\theta}$. Denote by $(\eta_n^{(h,\theta)}, \phi_n^{(h,\theta)})_{n=1}^{\infty}$ the eigensystem of the normal operator $F_{h,\theta}$. Define the function
\be\label{W}
W(h,\theta):=\left[ \sum_{n=1}^\infty \frac{|(\phi_n^{(h,\theta)}, u^\infty)_{L^2(\s)^2}|^2}{|\eta_n^{(h,\theta)}|}    \right]^{-1}.
\en
\begin{itemize}
\item[(i)] If $h\in [L_z, 2R]$, then the operator equation (\ref{F*}) is uniquely solvable, with the solution given by
\ben
g_{h,\theta}(\hat{x})=\sum_{n=1}^\infty \frac{(\phi^{h,\theta}_n, u^\infty)_{L^2(\s)^2}}{|\eta_n^{h,\theta}|^{1/2}} \phi^{h, \theta}_n(\hat{x}),\qquad \hat{x}\in \s.
 \enn Further, it holds that
 $||g_{h,\theta}||_{L^2(\s)^2}=W(h,\theta)>0$.
 \item[(ii)]  If $h\in (0, l_z]$, then the operator equation (\ref{F*}) has no solution in $L^2(\s)^2$ and $W(h,\theta)=0$.
 \item[(iii)] Let $h\in(l_z, L_z)$, and let $u$ be the total field corresponding to the scatterer $D$. If $u$ can be analytically extended from $D^c$ to $D\backslash\overline{B_{h,\theta}}\neq\emptyset$, then (\ref{F*}) is uniquely solvable and $W(h, \theta)>0$. Otherwise, we have $W(h, \theta)=0$.
\end{itemize}
\end{theorem}
\begin{proof}
Let $G_{h,\theta}: H^{1/2}(\partial B_{h,\theta})^2\rightarrow L^2(\s)^2$ be the data-to-pattern operator defined by $G_{h,\theta}(f)=v^\infty$, where $v^{\infty}$ is the far-field pattern of the scattered field $v^{sc}$ to the boundary value problem
\ben
\LL_\omega v^{sc}=0\quad\mbox{in}\quad \R^2\backslash\overline{B_{h,\theta}},\qquad v^{sc}=f\in H^{1/2}(\partial B_{h,\theta})^2.
\enn It is well known from \cite{AK, Arens} that the ranges of $G_{h,\theta}$ and $(F_{h,\theta}^*F_{h,\theta})^{1/4}$ coincide. In case (i), it is easy to see
\be\label{range}
u^{\infty}=G_{h,\theta}(f)\in \mbox{Range} (G_{h,\theta}),\quad\mbox{where}\quad f:=u^{sc}|_{\partial B_{h,\theta}}\in H^{1/2}(\partial B_{h,\theta})^{2},
\en and hence $u^\infty$ belongs to the range of $(F_{h,\theta}^* F_{h,\theta})^{1/4}$. By Picard's theorem, one obtains the results in the first assertion.
If $h\in(l_z, L_z)$ and $u$ can be analytically extended into $D\backslash \overline{B_{h,\theta}}$,  the scattered field $u^{sc}$ can be analytically continued to the domain $|x-z|>h$. This implies that we
have the relation (\ref{range}) again. In case (ii), it holds that $u^\infty\notin \mbox{Range}(G)$, because $u^{sc}$ cannot be analytically extended onto $\overline{D}$ as an entire
Kupradze radiating solution. The second part in the third assertion can be proved similarly.
\end{proof}
Write $z=z(\theta)\in \Gamma_R$ for $\theta\in[0,2\pi)$. Theorem 4.2 suggests the following indicator function for imaging the scatterer $D$:
\be\label{indicator}
I(y)=\left(\int_{0}^{2\pi}W(|y-z(\theta)|, \theta)\,d\theta\right)^{-1},\quad y\in B_R,
\en
where $W(h,\theta)$ is defined via (\ref{W}).
If $u$ cannot be extended into $D$ across any sub-boundary of $\partial D$, it holds that $W(|y-z(\theta)|,\theta)>0$ for all $y\in B_R$ such that $|y-z(\theta)|>L_z$, and
$W(|y-z(\theta)|,\theta)=0$ if $|y-z(\theta)|<L_z$. Therefore, the function $y\rightarrow W(|y-z(\theta)|,\theta)$ provides an estimate of the maximum distance between $z$ and $\partial D$ for fixed $\theta$. When the sampling variable $\theta$ varies in the whole interval $[0, 2\pi)$,
it is expected that $I(y)$ takes much larger values for $y\in D$ than for $y\in D^c$.
\begin{remark}\label{Re:obstacle}
If $D$ is a convex polygonal obstacle, it follows from Remark \ref{rem:sing}  (ii) that $u$ cannot be analytically continued across any corner of $\partial D$.
Hence, the indicator function (\ref{indicator}) could be used, in particular, for capturing a corner point of $\partial D$. If $D$ is a non-convex polygon, then the convex hull of $D$ can be efficiently recovered. The above scheme also applies to inverse scattering from penetrable scatterers and to inverse source problems. In the acoustic case, it was proved in \cite{BLS, PSV, ElHu2015, EH2018, LHY2018} that $u$ cannot be extended into $D$ across a strongly or weakly singular point of $\partial D$, that is, corners and weakly singular boundary points always scatter. Analogous results
in elastic scattering remain open, but similar conclusions can be expected.  Hence, the proposed numerical scheme can be utilized to recover boundary singular points of  penetrable and impenetrable scatterers.
\end{remark}

The authors in \cite{LS2018} proposed an extended linear sampling method for recovering $\partial D$ from a single acoustic far-field pattern. The idea there is to consider the solvability of the first kind integral equation
\be\label{LSM}
F_z g=u^\infty,
\en
where $F_z$ is the far-field operator corresponding to a sound-soft disk $|x-z|=a$ for some fixed $a>0$. Since the above equation is ill-posed, a regularization method must be used for solving (\ref{LSM}).  On the other hand, a multi-level sampling scheme was employed to find a proper radius of the sampling disk. In this paper, we have rigorously analyzed the solvability of the equation (\ref{F*}) within the framework of factorization method and have designed new sampling and imaging schemes, which avoid the multi-level sampling in
\cite{LS2018}.  In comparison with the linear sampling and factorization methods with all incident directions, the essential idea of \cite{LS2018} and this paper is to make
 use of the scattered data from an admissible set of known obstacles. Such an admissible set is taken as the set of
sampling disks  $B_{h,\theta}$ for all $h\in(0,2R)$, $\theta\in[0,2\pi)$ in this paper, and  was chosen to consist of $B_a(z)$ for all $z\in B_R$ with a fixed sampling radius $a>0$ in \cite{LS2018}.

The advantages of our inversion scheme can be summarized as follows. Firstly, the functions $W(h,\theta)$ and $I(y)$ involve only inner product calculations with low computational cost,
because
the spectrum $(\eta_n^{h,\theta}, \phi_n^{h,\theta})$ of the far-field operator for the disk $B_{h,\theta}$ can be obtained explicitly in elasticity.  For obstacles from other admissible set,
the spectrum of the corresponding far-field operator
can be obtained in advance.

Secondly, the spectral systems corresponding to a priori given obstacles from the admissible set can be replaced by other virtual systems which mathematically make sense.
For example, in the case of near-field measurement data and for time-dependent scattering problems, the original version of the factorization methods (see \cite{CHL18, Kirsch08})  involves physically non-meaningful incoming waves. However, the resulting far-field operators are still meaningful from the mathematical point of view. Therefore,
the revisited factorization method described here applies to these cases. 
There is also a variety in the choice of the shape and the boundary conditions of the scatterers in the admissible set.

{\color{rot}{Thirdly, the proposed inversion scheme may be applied to other shape identification problems for imaging penetrable and
impenetrable scatterers with a single incoming wave, including inverse source problems. However, the theoretical justification
of the non-analytical extension across a corner domain in linear elasticity seems more
 challenging than its acoustic counterpart.}}
Numerical results for inverse acoustic scattering problems with near-field and far-field data and further comparison with other sampling methods will be reported in forthcoming papers.

\section{Acknowledgements}
The first author gratefully acknowledges the support of the Computational Science Research Center in Beijing and the School of Mathematical Sciences
of the Fudan University in Shanghai during his stay in October of 2018.
The work of the second author is supported by NSFC grant No. 11671028 and NSAF grant No. U1530401.

\end{document}